\documentclass[final]{imanum}

\usepackage{overpic}
\usepackage{amsmath,amstext,amsbsy,amssymb,amsthm}
\usepackage{mathrsfs}
\usepackage{tikz}
\usepackage{todonotes}
\usepackage{courier}
\usetikzlibrary{positioning}
\usetikzlibrary{shapes,arrows}
\usepackage[fleqn,tbtags]{mathtools}
\usepackage{relsize}
\usepackage{url}
\usepackage{xcolor,colortbl}
\usepackage{tabularx}

\usepackage{pifont}

\renewcommand{\t}{\theta}
\newcommand{\dx}{\delta x}
\newcommand{\dt}{\delta \theta}
\newcommand{\ac}{\cos^{-1}}

\usepackage{listings}
\definecolor{mygreen}{rgb}{0.2,0.4,0.2}
\definecolor{mymauve}{rgb}{0.58,0,0.82}
\lstdefinestyle{customc}{
  language=Matlab,
  showstringspaces=false,
  basicstyle=\small\ttfamily,
  keywordstyle=\color{blue},
  commentstyle=\color{mygreen},
  identifierstyle=\color{black},
  stringstyle=\color{mymauve},
  deletekeywords={sech, zeros, plot, eye, sum, isnan, diff, hold, max, expm, sign, diag, ans, format, surf, axis, subplot, pi, sin, cos, eps, acos, rand, sqrt}
}
\jno{drnxxx}
\received{}
\revised{}
\title{A fast FFT-based discrete Legendre transform}
\shorttitle{A fast FFT-based discrete Legendre transform}
\author{Nicholas Hale\thanks{Department of Applied Mathematics, University of Stellenbosch, Stellenbosch, 7600, South Africa. (nickhale@sun.ac.za)} 
and Alex Townsend\thanks{Department of Mathematics, Massachusetts Institute of Technology, 77 Massachusetts Avenue,
Cambridge, MA 02139-4307. (ajt@mit.edu)}}
\shortauthorlist{N.\ Hale and A.\ Townsend}
\begin{document}
\maketitle

\begin{abstract}
{An $\mathcal{O}(N(\log N)^2/\log\!\log N)$ algorithm for the computation of the discrete Legendre 
transform and its inverse is described. The algorithm combines a recently developed fast transform for converting between Legendre and Chebyshev
coefficients with a Taylor series expansion for Chebyshev 
polynomials about equally-spaced points in the frequency domain. 
Both components are based on the FFT, and as an intermediate step we obtain an $\mathcal{O}(N\log N)$ algorithm 
for evaluating a degree $N-1$ Chebyshev expansion at an $N$-point Legendre grid. 
Numerical results are given to demonstrate performance
and accuracy.}
{discrete Legendre transform; Legendre polynomials; Chebyshev polynomials; fast Fourier transform}
\end{abstract}

\section{Introduction}
\label{sec:intro}

Given $N$ values $c_0,\ldots,c_{N-1}$, the discrete Legendre transform (DLT) or finite Legendre transform calculates the discrete sums 
\begin{equation} 
 f_k = \sum_{n=0}^{N-1} c_n P_n( x^{leg}_{k} ), \qquad 0\leq k\leq N-1,
\label{eq:fastTransform} 
\end{equation} 
where $P_n(x)$ denotes the degree $n$ Legendre polynomial and the Legendre nodes, $1>x^{leg}_{0}>\ldots>x^{leg}_{N-1}>-1$, are the 
roots of $P_N(x)$. 
The inverse discrete Legendre transform (IDLT), which computes $c_0,\ldots,c_{N-1}$ given $f_0,\ldots,f_{N-1}$,
can be expressed as
\begin{equation} 
 c_n = (n+\tfrac{1}{2})\sum_{k=0}^{N-1} w_k^{leg}f_kP_n( x^{leg}_{k}), \qquad 0\leq n\leq N-1,
\label{eq:fastTransformInv} 
\end{equation} 
where $w_0^{leg},\ldots w_{N-1}^{leg}$ are the Gauss--Legendre quadrature weights. It is the goal of this paper to describe fast algorithms
for computing the DLT and IDLT.

Expansions in Legendre polynomials can be preferred over Chebyshev expansions because Legendre polynomials are 
orthogonal in the standard $L^2$-norm. However, Legendre expansions are less convenient in practice because fast 
algorithms are not as readily available or involve a precomputational cost that inhibits applications employing 
adaptive discretizations. By deriving fast algorithms
for the DLT and IDLT we take some steps towards removing this barrier and allowing Legendre expansions and adaptive discretizations at Legendre 
nodes to become a practical tool in computational science.

Our approach makes use of the fast Chebyshev--Legendre transform described in~\cite[]{Hale_14_01}, which is based 
on carefully exploiting an asymptotic expansion of Legendre polynomials and 
the fast Fourier transform (FFT). 
Using the FFT has advantages and disadvantages. 
On the one hand, fast and industrial-strength 
implementations of algorithms for computing the FFT 
are ubiquitous. On the other hand, the FFT is restricted to equally-spaced samples and
is therefore not immediately applicable to situations like computing the DLT. In this paper we overcome 
the equally-spaced restriction by considering 
Legendre nodes as a perturbation of a Chebyshev grid 
and employing a truncated Taylor series expansion.
A similar approach for the discrete Hankel transform is described in~\cite[]{Townsend_15_01}.
\begin{figure}[t]
{\footnotesize
\begin{center}
\hspace*{-25pt}
\tikzstyle{block} = [rectangle, draw, fill=blue!0,
    text width=6em, text centered, node distance=2.4cm, rounded corners, minimum height=4em]
\tikzstyle{line} = [draw, thick, color=black, -latex']
\hspace*{18pt}
\begin{tikzpicture}[scale=2, node distance = 1cm, auto]
    \node [block,text width=6.3em, node distance=1cm] (chebvalues) {Values at Chebyshev points};
    \node [block, right = 3.3cm of chebvalues,text width=6.3em ,node distance=1cm] (legcoeffs) {Legendre coefficients};
    \node [block, right = 3.3cm of legcoeffs,text width=6.3em, node distance=1cm] (chebcoeffs) {Chebyshev coefficients};
    \node [block, below left = 1cm and 1.2cm of legcoeffs,text width=6.3em ,node distance=1cm] (Iserles) {Values in the complex plane};
    \node [block, below right = 1cm and 1.2cm of legcoeffs,text width=6.3em ,node distance=1cm] (Tygert) {Values at Legendre points};
\path ([yshift=2pt]chebvalues.east) edge[<->, thick,dashed] node[above=-1.5pt] {\cite[]{Potts_98_01}} ([yshift=2pt]legcoeffs.west);
\path ([yshift=-2pt]chebvalues.east) edge[<-, thick,dashed] node[below=.5pt] {\cite[]{Mori9901}} ([yshift=-2pt]legcoeffs.west);
\path ([yshift=-2pt]legcoeffs.east) edge[<->, thick,dashed]  node[above=-1.5pt] {\cite[]{Hale_14_01}} ([yshift=-2pt]chebcoeffs.west);
\path ([yshift=-2pt]legcoeffs.east) node[above right=5.5pt and 20pt] {\cite[]{Keiner_08_01}} ([yshift=-2pt]chebcoeffs.west);
\path ([yshift=-2pt]legcoeffs.east) node[above right=13pt and 4pt] {\cite[]{Alpert_91_01}}([yshift=-2pt]chebcoeffs.west);
\path (chebvalues.north) edge[ out=30, in=150, looseness=0.2, loop, distance=1.2cm, <->, thick,dashed] node[above=0pt] {DCT} (chebcoeffs.north);
\path [draw, -latex',thick,<-,dashed] ([xshift=-2pt]legcoeffs.south) |- node[below left = .5pt and 0pt] {\cite[]{Iserles_11_01}} (Iserles);

\path [draw, -latex',thick,<->,dashed] ([xshift=2pt]legcoeffs.south) |- node[below right = .5pt and 3pt] {\cite[]{Potts_03_01}} (Tygert);
\path ([xshift=2pt]legcoeffs.south) node[below right = 53pt and .5pt] {\cite[]{Keiner_09_01}} (Tygert);
\path ([xshift=2pt]legcoeffs.south) node[below right = 60.75pt and 1pt] {\cite[]{Tygert_10_01}} (Tygert);

\path [draw, -latex',thick,->, red] ([xshift=0pt]chebcoeffs.south) |- node[above right = 16pt and 0pt] {NDCT} ([yshift=0pt]Tygert);
\path ([xshift=0pt]chebcoeffs.south) |- node[above right = 26pt and -79pt] {\cite[]{Dutt_93_01,Dutt_95_01}} ([yshift=0pt]Tygert);
\path ([xshift=0pt]chebcoeffs.south) |- node[above right = 18.5pt and -72pt] {\cite[]{Fenn_05_01}} ([yshift=0pt]Tygert);

\path [draw, -latex',thick,->, red] ([xshift=1pt, yshift=1pt]legcoeffs.south east) -- node[above right = -2pt and 0pt] {DLT} ([xshift=1pt, yshift=1pt]Tygert.north west);
\path [draw, -latex',thick,<-, red] ([xshift=-1pt, yshift=-1pt]legcoeffs.south east) -- node[above left = -8pt and 0pt] {IDLT} ([xshift=-1pt, yshift=-1pt]Tygert.north west);
\end{tikzpicture}
\end{center}
}
\caption[]{Some existing fast algorithms related to the DLT and IDLT. Solid lines depict transforms discussed in this 
paper. The ``non-uniform discrete cosine transform'' (NDCT) is described in Section~\ref{sec:dct}, the DLT in Section~\ref{sec:dlt}, 
and the IDLT in Section~\ref{sec:idlt}. This table is far from complete. Furthermore, some of the papers referenced belong
to more than one of the connecting  lines; for example, \cite[]{Keiner_09_01} 
computes the DLT by combining
a transformation from Legendre coefficients to Chebyshev coefficients with an NDCT. We have tried to include as many of the 
key papers as possible, without overcomplicating the diagram. 
}\label{fig:existingmethods}
\end{figure}
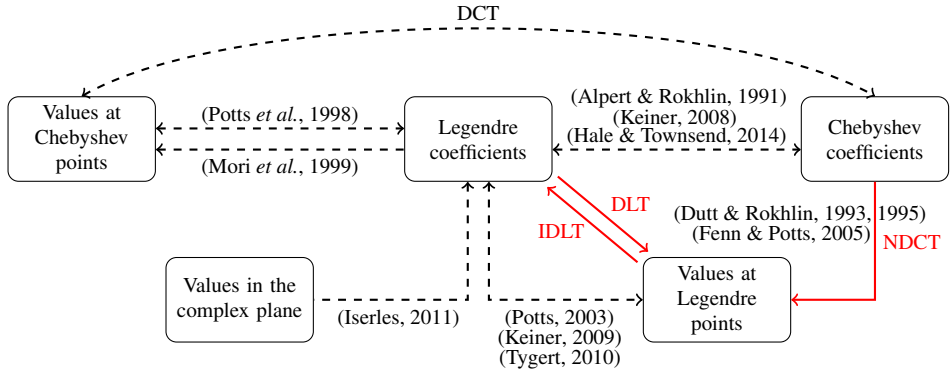%

Examples of existing approaches that may be used or adapted to compute the DLT include butterfly schemes~\cite[]{Neil_10_01}, 
divide-and-conquer approaches~\cite[]{Keiner_08_01, Tygert_10_01},
and the oversampled non-uniform discrete cosine transform~\cite[]{Potts_03_01, Fenn_05_01, Keiner_09_02}.
The algorithm 
presented here has three advantages: (1) It is simple (the main ingredients are
just the FFT and Taylor approximations), (2) There is essentially no precomputational cost; and, (3) There
are no algorithmic parameters for the user to tweak as they are all optimally determined by analysis. 
In particular, these final two points make the algorithm presented in this paper well-suited to applications where $N$
is not known in advance or when there are multiple accuracy goals.

It is interesting to note that before the work of~\cite{Glaser_07_01} 
there was no known stable method for computing the 
Legendre nodes 
in fewer than $\mathcal{O}(N^2)$ operations.\footnote{The current state of the art for computing the Legendre nodes is due to~\cite{Bogaert_14_01}, 
whose $\mathcal{O}(N)$ algorithm can be rapidly compute millions of nodes to 16-digits of accuracy in under a second. 
For a discussion of the history of computing Gauss--Legendre quadrature nodes and weights, see \cite[]{Townsend_15_02}.}  Therefore, prior to 2007, it seems
likely that all DLT algorithms had a precomputational cost of at least $\mathcal{O}(N^2)$. Nowadays, we believe that the algorithms described above, with the exception of~\cite[]{Neil_10_01}, can be considered to have a lower precomputational cost. 


The paper is structured as follows: In the next section we develop 
an $\mathcal{O}(N\log N)$ algorithm for evaluating a Chebyshev series expansion 
at Legendre nodes (we refer to this as a ``non-uniform discrete cosine 
transform'' or NDCT). In Section~\ref{sec:dlt} we combine this with the fast 
Chebyshev--Legendre transform from~\cite[]{Hale_14_01} to 
obtain an $\mathcal{O}(N(\log N)^2/\log\!\log N)$ algorithm for computing the DLT.
A similar algorithm for the IDLT is described in Section~\ref{sec:idlt}.
Throughout we use the following notation: A column vector with entries $v_0,\ldots,v_{N-1}$ 
is denoted by $\underline{v}$, a row vector by $\underline{v}^T$, a diagonal 
matrix with entries $v_0,\ldots,v_{N-1}$ by $D_{\underline{v}}$, and the $N\times N$ 
matrix with $(j,k)$ entry $P_k(x_{j})$ by $\mathbf{P_N}(\underline{x})$.%

%
\section{A non-uniform discrete cosine transform}
\label{sec:dct}
Before considering the DLT~\eqref{eq:fastTransform} we describe 
an $\mathcal{O}(N\log N)$ algorithm for evaluating a Chebyshev expansion at Legendre nodes,
\begin{equation}
 f_k = \sum_{n=0}^{N-1} c_n T_n( x^{leg}_{k} ), \qquad 0\leq k\leq N-1.
 \label{eq:nuDCT} 
\end{equation} 
This will become an important step in the 
DLT. Here, $T_n(x) = \cos(n\ac x )$ is the degree $n$ Chebyshev polynomial 
of the first kind and the sums in~\eqref{eq:nuDCT} may be rewritten as
\begin{equation} 
 f_k = \sum_{n=0}^{N-1} c_n \cos(n \t^{leg}_{k} ), \qquad 0\leq k\leq N-1,
 \label{eq:nuDCT2}
\end{equation} 
where $x^{leg}_k = \cos(\t^{leg}_k)$. 
If the Legendre nodes in~\eqref{eq:nuDCT2} are replaced by 
the Chebyshev points of the first kind, i.e., 
\begin{equation}
 x_{k}^{cheb_1} = \cos(\t^{cheb_1}_k) = \cos\left(\frac{(k+\tfrac{1}{2})\pi}{N}\right),\qquad 0\leq k\leq N-1,
\label{eq:chebpts1}
 \end{equation}
then~\eqref{eq:nuDCT} can be computed in $\mathcal{O}(N\log N)$ 
operations via a diagonally-scaled discrete cosine transform of type III (DCT-III)~\cite[]{Gentleman_72_01}.
Since~\eqref{eq:nuDCT2} has a similar form to a DCT, but with 
non-equally-spaced samples $\t^{leg}_k$, 
we refer to~\eqref{eq:nuDCT} as a ``{non-uniform discrete cosine transform}'' (NDCT).

Our algorithm to compute the NDCT is summarized as follows: consider 
the transformed Legendre nodes, 
$\t_0^{leg}, \ldots, \t_{N-1}^{leg}$,
as a perturbation of an equally-spaced grid, 
$\t_0^{\ast},\ldots, \t_{N-1}^{\ast}$, i.e., 
\[
 \t^{leg}_k = \t^{\ast}_k + \dt_k, \qquad 0\le k\le N-1,
\]
and then approximate each $\cos(n\t^{leg}_k)$ term in~\eqref{eq:nuDCT2}
by a truncated Taylor series expansion about $\t_k^{\ast}$. 
If $|\dt_k|$ 
is small (see Section~\ref{subsec:legpts}) then only a few terms in the Taylor 
series expansion are required. Moreover, since the
$\t_0^{\ast},\ldots, \t_{N-1}^{\ast}$
are equally-spaced, each term in 
the series can be 
computed in $\mathcal{O}(N\log N)$ operations via a DCT (see Section~\ref{subsec:dct9}).
The Legendre nodes, both $x_0^{leg}, \ldots, x^{leg}_{N-1}$ and $\t_0^{leg}, \ldots,\t^{leg}_{N-1}$, can be computed in $\mathcal{O}(N)$ operations 
using the algorithm described in \cite[]{Bogaert_14_01}.%
%
\subsection{Computing an NDCT via a Taylor series expansion}\label{subsec:taylor}
As in~\cite[]{Hale_14_01} we find it convenient to work in the $\theta = \ac x$ variable. 
There are three reasons for this:
\begin{itemize}
\itemsep0em
\item[(1)] The quantities $\theta^{leg}_k$ can be computed more accurately than $x^{leg}_k$~\cite[]{Swarztrauber_03_01}. 
\item[(2)] The function $\ac x$ is ill-conditioned near $x=\pm 1$. In particular, since $\frac{d}{dx}\ac x = (1-x^2)^{-1/2}$, $1-(x_0^{leg})^2= \mathcal{O}(N^{-2})$, and $1-(x_{N-1}^{leg})^2 = \mathcal{O}(N^{-2})$, one can expect rounding errors that grow like $\mathcal{O}(N)$ when evaluating $\ac (x_0^{leg})$ and $\ac (x_{N-1}^{leg})$. Furthermore, when evaluating $T_{N-1}(\ac (x_0^{leg}))$ and $T_{N-1}(\ac (x_{N-1}^{leg}))$ this rounding error can increase to $\mathcal{O}(N^2)$;
\item[(3)] The Taylor series of $T_n(\cos(\t+\dt)) = \cos(n(\t+\dt))$ about $\t$ involves simple trigonometric terms and 
is more convenient to work with than the ultraspherical polynomials that arise in a Taylor series of $T_n(x+\dx)$ about $x$.
\end{itemize}
The Taylor series expansion of $T_n(\cos(\theta+\dt))$ about $\theta\in[0, \pi]$ can be expressed as  
\begin{equation}
 \cos(n(\t+\dt)) = \sum_{\ell=0}^\infty \cos^{(\ell)}(n\t)\frac{(\dt)^\ell}{\ell!} = \sum_{\ell=0}^\infty (-1)^{\lfloor (\ell+1)/2\rfloor}\Phi_\ell(n\t)\frac{(n\dt)^\ell}{\ell!},
\label{eq:taylor}
 \end{equation}
where
\[
 \Phi_\ell(\t) = \begin{cases}\cos(\t), & \ell \text{ even,}\\ \sin(\t), & \ell \text{ odd.} \end{cases} 
\]
If $|\dt|$ is small then a good 
approximation to $T_{n}(\cos(\t+\dt))$ can be obtained by truncating 
this series after just a handful of terms. The following lemma determines the 
accuracy we 
can expect from taking the first $L$ terms.
\begin{lemma}\label{en:RLx_bound}
For any $L\geq1$ and $n\geq 0$, 
 \[
  R_{L,n,\dt} := \max_{\t\in[0, \pi]}\left|\cos(n(\t+\dt)) - \sum_{\ell=0}^{L-1} \cos^{(\ell)}(n\t)\frac{(\dt)^\ell}{\ell!}\right|\leq \frac{\left(n|\dt|\right)^{L}}{L!}.
 \]
\end{lemma}
\begin{proof} 
 By the mean-value form of the remainder we have 
 \[
  R_{L,n,\dt} \leq  \max_{\hat\t\in[0,\pi]} |\cos^{(L)}(n\hat\t)|\frac{|\dt|^{L}}{L!}\leq \frac{(n|\dt|)^{L}}{L!},
 \]
as required.
\end{proof}

For $0\leq k\leq N-1$, we write $\t_k^{leg} = \t_k^{*} + \dt_k$ and substitute the first $L$ terms 
of~\eqref{eq:taylor} into~\eqref{eq:nuDCT} to obtain,
\begin{eqnarray}\label{eq:subsTaylor}
 f_{k,L,\dt_k} &=& \sum_{n=0}^{N-1}c_n\left(\sum_{\ell=0}^{L-1} (-1)^{\lfloor (\ell+1)/2\rfloor}\Phi_\ell(n\t_k^{*})\frac{(n\dt_k)^\ell}{\ell!}\right)\\
 &=& \sum_{\ell=0}^{L-1} (-1)^{\lfloor (\ell+1)/2\rfloor}\frac{\dt_k^\ell}{\ell!}\left(\sum_{n=0}^{N-1}(n^\ell c_n)\Phi_\ell(n\t_k^{*})\right),
 \label{eqn:terms} 
\end{eqnarray}
where the second equality follows by interchanging the order of the summations. 
\begin{corollary}\label{cor:fkbnd}
If $f_k$ is as in~\eqref{eq:nuDCT} and $f_{k,L,\dt_k}$ as in~\eqref{eq:subsTaylor}, then for any $L\geq1$, 
\[
 |f_k - f_{k,L,\dt_k}| \leq \sum_{n=0}^{N-1}|c_n|R_{L,n,\dt_k} \leq \frac{\left(N|\dt_k|\right)^{L}}{L!}||\underline{c}||_1, \qquad 0\leq k\leq N-1.
\] 
\end{corollary}
\begin{proof}
 This follows immediately from applying Lemma~\ref{en:RLx_bound} to each of the terms in~\eqref{eq:nuDCT}.
\end{proof}

Hence, if we approximate $\t_0^{leg},\ldots,\t_{N-1}^{leg}$ by points 
$\t^\ast_0, \ldots, \t_{N-1}^\ast$ such that
\begin{itemize}\itemsep0em
\item $\displaystyle\max_{0\le k\le {N-1}}|{\t}_k^{leg}-{\t}_k^{*}|$ is sufficiently small (see Section~\ref{subsec:legpts}); and, 
\item $\t^\ast_0, \ldots, \t_{N-1}^\ast$ is equally spaced,
\end{itemize}
then the number of terms, $L$, required to achieve an accuracy of $|f_k - f_{k,L,\dt_k}|\leq \varepsilon||\underline{c}||_1$ is small. 
Moreover, the inner sums of~\eqref{eqn:terms} can be computed 
via discrete cosine and sine transforms (see Section~\ref{subsec:dct9}), which  
can be computed in $\mathcal{O}(N\log N)$ via an FFT. This gives us the 
foundations of a fast algorithm.

\subsection{Legendre nodes as a perturbation of a Chebyshev-like grid}\label{subsec:legpts}
The Chebyshev points of the first kind, ${x}^{cheb_1}_0, \ldots, {x}^{cheb_1}_{N-1}$, are
an approximation to ${x}_0^{leg}, \ldots, {x}_{N-1}^{leg}$ which satisfy the requirements in the two bullet points above. In particular, letting $\theta_k^{cheb_1} = \ac(x^{cheb_1}_k)$, one can readily show that~\cite[(18.16.3)]{NISTHandbook}
\[
 |{\t}^{leg}_k - {\t}^{cheb_1}_k| \leq \frac{\pi}{2(N+1)}, \qquad 0\leq k\leq N-1.
 \]
However, this bound is a little weak. A better bound is given in Lemma~\ref{lemma:szego}.

Another possibility is to consider the leading order term of the asymptotic expansion of Legendre polynomials of 
large degree~\cite[]{Stieltjes_1890_01}:
\[
 P_N(\cos \theta) \sim \sqrt{\frac{4}{\pi}}\frac{\Gamma(N+1)}{\Gamma(N+\tfrac{3}{2})}\frac{\cos\!\left((N+\tfrac{1}{2})\theta-\tfrac{\pi}{4}\right)}{\left(2\sin\theta\right)^{1/2}},\qquad N\rightarrow\infty.
\]
The zeros of this leading term are 
\[
    \theta^{cheb_*}_{k} = \frac{(k+\tfrac{3}{4})\pi}{N+\tfrac{1}{2}}, \qquad 0\leq k\leq N-1,
\]
which are also equally-spaced and provide us with an approximation to Legendre nodes. 
In fact, multiplying the numerator and denominator by a factor
of $2$ and comparing to~\eqref{eq:chebpts1} we see that these points correspond precisely to 
the odd terms (i.e., when the index $k$ is odd) in the ($2N+1$)-point
Chebyshev grid of the first kind. 

The following lemma
shows how closely $\t^{cheb_\ast}_k$ and $\t^{cheb_1}_k$ approximate $\t^{leg}_k$, 
and Figure~\ref{fig:dx} demonstrates that the derived bounds are tight.

\begin{lemma}\label{lemma:szego}
 Let $\theta_{0}^{leg},\ldots,\theta_{N-1}^{leg}$ denote the $N$ roots of $P_N(\cos\theta)$. Then, 
 \begin{equation}\label{eqn:szego}
 \max_{0\leq k\leq N-1}\left|\theta_{k}^{leg} - \theta^{cheb_*}_{k} \right| \leq \frac{1}{3\pi(2N+1)} \leq \frac{1}{6\pi N},
 \end{equation}
 and 
 \begin{equation}\label{eqn:szegob}
  \max_{0\leq k\leq N-1}\left|{\t}_k^{leg} - {\t}_k^{cheb_1}\right| \leq \frac{3\pi^2+2}{6\pi(2N+1)} \leq \frac{0.83845}{N}.
 \end{equation}
\end{lemma}
\begin{proof} 
We first consider $|\theta_{k}^{leg} - \theta^{cheb_*}_{k} |$
and restrict our attention to $0\leq k\leq \lfloor N/2\rfloor-1$. From \cite[Thm.\ 6.3.2]{Szego_36_01} we know that $\theta^{cheb_*}_k < \theta_k^{leg}$. 
Moreover, by~\cite[Thm.\ 3.1 \& eqn.\ (3.2)]{Brass_97_01} and $\cot(x)<1/x$, we have 
\[
 \theta_{k}^{leg}-\theta^{cheb_*}_{k} \leq \frac{1}{2(2N+1)^2}\cot\theta^{cheb_*}_k \leq \frac{4N+2}{2(2N+1)^2(4k+3)\pi}\leq \frac{1}{3(2N+1)\pi},\qquad 0\leq k\leq \lfloor N/2\rfloor-1. 
\]
By symmetry of the ${\t}_k^{leg}$ and ${\t}^{cheb_*}_k$, the 
same bound holds for $k > \lfloor N/2\rfloor-1$ when we take the modulus. 

The bound on $|{\t}_k^{leg} - {\t}_k^{cheb_1}|$ follows from the triangle inequality. That is, 
for $0\leq k\leq N-1$, we have
\[
 \left|{\t}_k^{leg} - {\t}_k^{cheb_1}\right| \leq \left|\theta_{k}^{leg} - \theta^{cheb_*}_{k} \right| + \left|\theta_{k}^{cheb_1} - \theta^{cheb_*}_{k} \right|\leq \frac{1}{3\pi(2N+1)} + \frac{\pi(2k-N+1)}{2N(2N+1)}.   
\]
Since $k\leq N-1$, we obtain 
\[
 \left|{\t}_k^{leg} - {\t}_k^{cheb_1}\right|\leq \frac{1}{3\pi(2N+1)} + \frac{\pi(N-1)}{4N(N+1/2)} \leq \frac{3\pi^2+2}{6\pi(2N+1)} \leq \frac{0.83845}{N}, \qquad 0\leq k\leq N-1.
\]
\end{proof}

\begin{figure}[t]
\vspace{-4cm}
\begin{center}
\begin{overpic}[width=.5\textwidth]{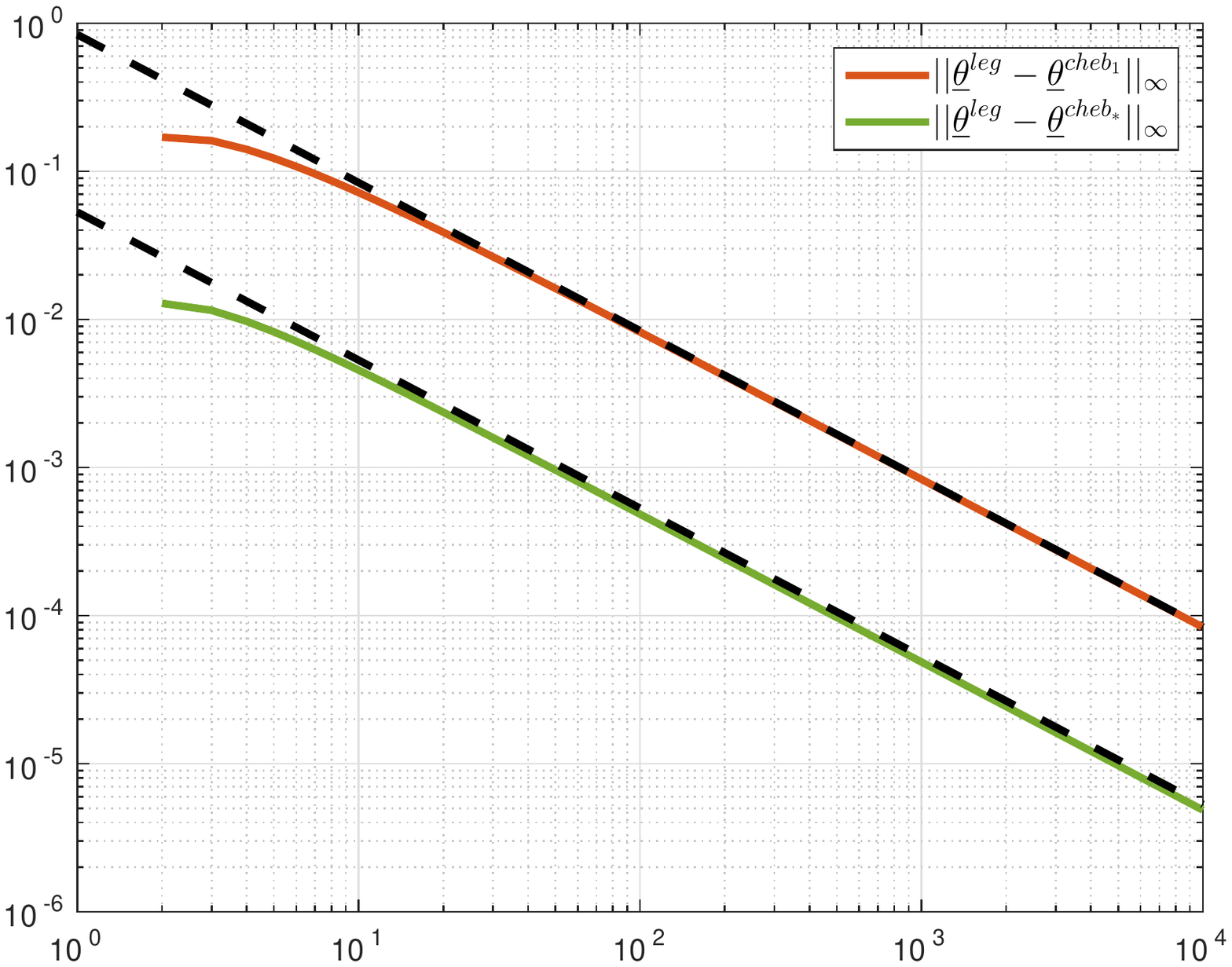} 
\put(35,0) {\footnotesize{$N$}}
\put(30,41) {\footnotesize{\rotatebox{-28.5}{$\|\underline{\t}^{leg} - \underline{\t}^{cheb_1}\|_\infty\leq 0.83845N^{-1}$}}}
\put(30,32) {\footnotesize{\rotatebox{-28.5}{$\|\underline{\t}^{leg} - \underline{\t}^{cheb_*}\|_\infty\leq (6\pi N)^{-1}$}}}
\end{overpic}
\caption{The maximum distance between corresponding entries of $\underline{\theta}^{leg}$ and $\underline{\theta}^{cheb_1}$ (upper solid line) and 
$\underline{\theta}^{leg}$ and $\underline{\theta}^{cheb_*}$ (lower solid line) as a function of $N$. The dashed lines depict the bounds derived in 
Lemma~\ref{lemma:szego}, which appear to be tight.}
\label{fig:dx}
\end{center}
\end{figure}

\begin{corollary}\label{cor:fkbounds}
Choosing $\theta_k^\ast = \t^{cheb_1}_k$ or $\theta_k^\ast = \t^{cheb_\ast}_k$ for $0\le k\le {N-1}$ in~\eqref{eqn:terms}, we have
\begin{equation}\label{eq:C_La}
\max_{0\leq k\leq N-1}|f_k - f_{k,L,\dt^{cheb_1}_k}|\leq \frac{(0.83845)^L}{L!}||\underline{c}||_1
\end{equation}
or
\begin{equation}\label{eq:C_Lb}
\max_{0\leq k\leq N-1}|f_k - f_{k,L,\dt^{cheb_\ast}_k}|\leq \frac{1}{(6\pi)^LL!}||\underline{c}||_1,
\end{equation}
respectively.
\end{corollary}
\begin{proof}
Follows immediately from combining Corollary~\ref{cor:fkbnd} and Lemma~\ref{lemma:szego}.
\end{proof}

We note that in~\eqref{eq:C_La} and~\eqref{eq:C_Lb} the terms multiplying $\|\underline{c}\|_1$ are independent of $N$, 
and so the number of terms required in~\eqref{eqn:terms} to obtain a given precision is bounded independently of $N$. 
Table~\ref{tbl:L} shows these terms for $1\leq L\leq 18$. 
Double precision is obtained for $L \gtrsim9$ terms when $\t_k^\ast = \t^{cheb_\ast}_k$ and $L \gtrsim17$ terms when $\t_k^\ast = \t^{cheb_1}_k$. 
\begin{table}[h!]
\footnotesize\bgroup\def\arraystretch{1.1}\tabcolsep=1.5pt\hspace*{-15pt}%
\begin{tabular}{ | c | c | c | c | c | c | c | c | c | c |}%
\hline
  $\quad L\quad$   & 1 & 2 & 3 & 4 & 5 & 6 & 7 & 8& 9\\\hline
  $0.83845^L/L!$ & $8.4\times10^{-1}$ & $3.5\times10^{-1}$ & $9.8\times10^{-2}$ & $2.1\times10^{-2}$ & $3.5\times10^{-3}$ & $4.8\times 10^{-4}$ & $5.8\times 10^{-5}$ & $6.1\times10^{-6}$ & $5.6\times10^{-7}$\\
  $((6\pi)^LL!)^{-1}$ & $5.3\times10^{-2}$ &  $1.4\times10^{-3}$ &  $2.5\times10^{-5}$ & $3.3\times10^{-7}$ & \cellcolor{gray!25}$3.5\times10^{-9}$ & $3.1\times10^{-11}$ & $2.4\times10^{-13}$ & $1.6\times10^{-15}$ & \cellcolor{gray!25}$9.2\times10^{-18}$\\\hline\hline
  $\quad L\quad$   & 10 & 11 & 12 & 13 & 14 & 15 & 16 & 17 & 18\\\hline
  $0.83845^L/L!$ & $4.7\times10^{-8}$ & \cellcolor{gray!25}$3.6\times10^{-9}$ & $2.5\times10^{-10}$ & $1.6\times10^{-11}$ & $9.7\times10^{-13}$ & $5.4\times10^{-14}$ & $2.9\times10^{-15}$ & \cellcolor{gray!25}$1.4\times10^{-16}$ & $6.6\times10^{-18}$\\\hline
  $((6\pi)^LL!)^{-1}$ & $4.9\times10^{-20}$ &  $2.3\times10^{-22}$ & $1.0\times10^{-24}$ & $4.2\times10^{-27}$ & $1.6\times10^{-29}$ & $5.7\times10^{-32}$ & $1.9\times10^{-34}$ & $5.9\times10^{-37}$ & $1.7\times10^{-39}$\\\hline
\end{tabular}
\egroup
\caption{The terms multiplying $\|\underline{c}\|_1$ in~\eqref{eq:C_La} and~\eqref{eq:C_Lb} for $1\leq L\leq 18$. 
The shaded boxes show the smallest value of $L$ for which these fall below single and double precision. 
}\label{tbl:L}
\end{table}
\subsection{Computing the discrete cosine and sine transforms}\label{subsec:dct9}
To complete our fast algorithm for the NDCT we require a fast way to compute the sums contained inside the parenthesis of~\eqref{eqn:terms}.
In particular, writing~\eqref{eqn:terms} 
in vector form, we have
\begin{equation}\label{eqn:lincomb}
 \underline{f}_{\,k,L} = \sum_{\substack{\ell=0\\even}}^{L-1}\frac{(-1)^{\lfloor(\ell+1)/2\rfloor}}{\ell!}D_{\underline{\dt}}^\ell {\rm DCT}(\underline{c}^{[\ell]})
 + \sum_{\substack{\ell=1\\odd}}^{L-1} \frac{(-1)^{\lfloor (\ell+1)/2\rfloor}}{\ell!}D_{\underline{\dt}}^\ell {\rm DST}(\underline{c}^{[\ell]}),
\end{equation}
where $[\underline{c}^{[\ell]}]_n = n^\ell c_n$, $n = 0,\ldots, N-1$.
When $\underline{\t}^\ast = \underline{\t}^{cheb_1}$ the DCT and DST take the form 
\[
    \sum_{n=0}^{N-1} n^\ell c_n \cos\left(\frac{n(k+\frac{1}{2})\pi}{N}\right), \qquad  \sum_{n=0}^{N-1} n^\ell c_n \sin\left(\frac{n(k+\frac{1}{2})\pi}{N}\right), \qquad 0\leq k \leq N-1,
\]
which are readily related to the DCT-III and DST-III, respectively. When $\underline{\t}^\ast = \underline{\t}^{cheb_\ast}$ they become
\[
    \sum_{n=0}^{N-1} n^\ell c_n \cos\left(\frac{n((2k+1)+\frac{1}{2})\pi}{2N+1}\right), \quad  \sum_{n=0}^{N-1} n^\ell c_n \sin\left(\frac{n((2k+1)+\frac{1}{2})\pi}{2N+1}\right), \quad 0\leq k \leq N-1,
\]
which are equivalent to the odd terms of a DCT-III and DST-III of length $2N+1$. 
Hence, both may be evaluated in $\mathcal{O}(N\log N)$ operations.

\begin{remark}\label{remark1}

Although, for an accuracy of double precision, the $\underline{\t}^{cheb_*}$ points 
require around half the number of terms in the Taylor series compared to $\underline{\t}^{cheb_1}$
(see Table~\ref{tbl:L}), we see from above that each term will be roughly twice as
expensive to evaluate. Hence, we expect little difference between the two approaches when there is a 
16-digit accuracy goal.
\end{remark}

\begin{remark}\label{remark:dct}
    The numerical results in this paper are computed in MATLAB. Although MATLAB
    uses FFTW for FFT computations \cite[]{FFTW}, it does not allow access to FFTW's DCT and DST routines (\lstinline[basicstyle=\ttfamily\footnotesize]{REDFT01}
    and \lstinline[basicstyle=\ttfamily\footnotesize]{RODFT01}, respectively.).\footnote{
    The MATLAB signal processing toolbox has a \lstinline[basicstyle=\ttfamily\footnotesize]{dct()} function, but this is computed via an FFT of double the length.} 
    Instead, we use the DCT and DST codes implemented
    in Chebfun~\cite[]{chebfun},
    which compute the DCT and DST via a complex-valued FFT 
    of double the length. Based on our experiments, we believe this to be a factor of 2-4
    slower than calling \lstinline[basicstyle=\ttfamily\footnotesize]{REDFT01} and \lstinline[basicstyle=\ttfamily\footnotesize]{RODFT01} directly.
\end{remark}

\begin{remark}
 This section described a Taylor-based NDCT, which is closely related to the Taylor-based 
 NFFT described by~\cite{Kunis_08_01}. Alternative approaches based on 
 oversampling and low-pass filters are described in \cite{Dutt_93_01} and~\cite{Potts_03_01}. 
 Although for general point distributions the oversampled NFFT and NDCT are considered faster than the Taylor-based approaches~\cite[]{Ware_98_01}
 we choose to use a Taylor-based NDCT because (a) it is simple and (b) Lemma~\ref{lemma:szego}, 
 provides a precise bound on the number of terms required in the expansion for 
 a given accuracy (see Table~\ref{tbl:L}). Since the required number of terms is small
 we expect the Taylor-based method to be competitive in this particular case.
\end{remark}


\subsection{Direct approach}\label{subsec:rec}
To test the fast algorithm above we construct $\mathbf{T}_N(\underline{x}^{leg})$ 
via the three-term recurrence relation:
\[
T_0(x) = 1, \qquad T_1(x) = x, \qquad T_{n+1}(x) = 2xT_n(x) - T_{n-1}(x), \quad n\geq 1.
\]
Although this approach requires $\mathcal{O}(N^2)$ operations to compute $\mathbf{T}_N(\underline{x}^{leg})\underline{c}$,
the memory cost can easily be reduced to $\mathcal{O}(N)$ by computing the matrix-vector product as a recurrence.
Hereinafter, we refer to this as the {\em direct approach}.

We note that one could instead compute the NDCT using the closed form expression of 
$T_n(\cos\t)$ to construct $\mathbf{T}_N(\underline{x}^{leg}) = \cos\left(\underline{\t}^{leg}[0, 1, \ldots, N-1]\right)$ and then evaluate the 
matrix-vector product $\underline{f} = \mathbf{T}_N(\underline{x}^{leg})\underline{c}$.
However, such a closed form is not available when we come to compute $\mathbf{P}_N(\underline{x}^{leg})$ in the next section.


\subsection{Numerical results for the NDCT}\label{subsec:NumericalResults}
All the numerical results were performed on a 3.40GHz Intel Core i7-2600 PC
with MATLAB 2015a in Ubuntu Linux. The vector of transformed nodes $\underline{\t}^{leg}$ is computed with 
the \lstinline{legpts()} command in Chebfun \cite[v5.2]{chebfun}
which uses Bogaert's algorithm~\cite[]{Bogaert_14_01}. As discussed in Remark~\ref{remark:dct},
DCTs and DSTs are computed using the \lstinline{chebfun.dct()} and \lstinline{chebfun.dst()} routines. MATLAB codes
for reproducing all of the results contained within this paper are available online \cite[]{Hale_15_01}.

As a first test we take randomly distributed vectors $\underline{c}$ with various rates of 
decay and compare the accuracy of both the direct and FFT-based approaches against an extended precision computation
(Figure~\ref{fig:ndct_test01}).\footnote{
In particular, the vector corresponding to, say,  $N = 50$ with $\mathcal{O}(n^{-0.5})$ decay can be reproduced exactly 
by the MATLAB code \lstinline[basicstyle=\ttfamily\footnotesize]{rng(0,'twister'); c = rand(50,1)./sqrt(1:50).';}.}
The results for $\underline{\t}^{cheb_1}$ and $\underline{\t}^{cheb_\ast}$
in the FFT-based approach are virtually indistinguishable so we show only the former.
We find the FFT-based approaches are more accurate than the direct approach, despite our algorithm 
involving many significant approximations.  In many applications the Chebyshev expansion in~\eqref{eq:nuDCT} 
is derived by an approximation of a smooth function. In this setting, if the function is H\"{o}lder 
continuous with $\alpha>1/2$, we observe that that the FFT-based algorithm has essentially no error growth with $N$.
\begin{figure}[t]
\vspace*{-4cm}
  \centering
  \scriptsize
    \begin{overpic}[width=.49\textwidth]{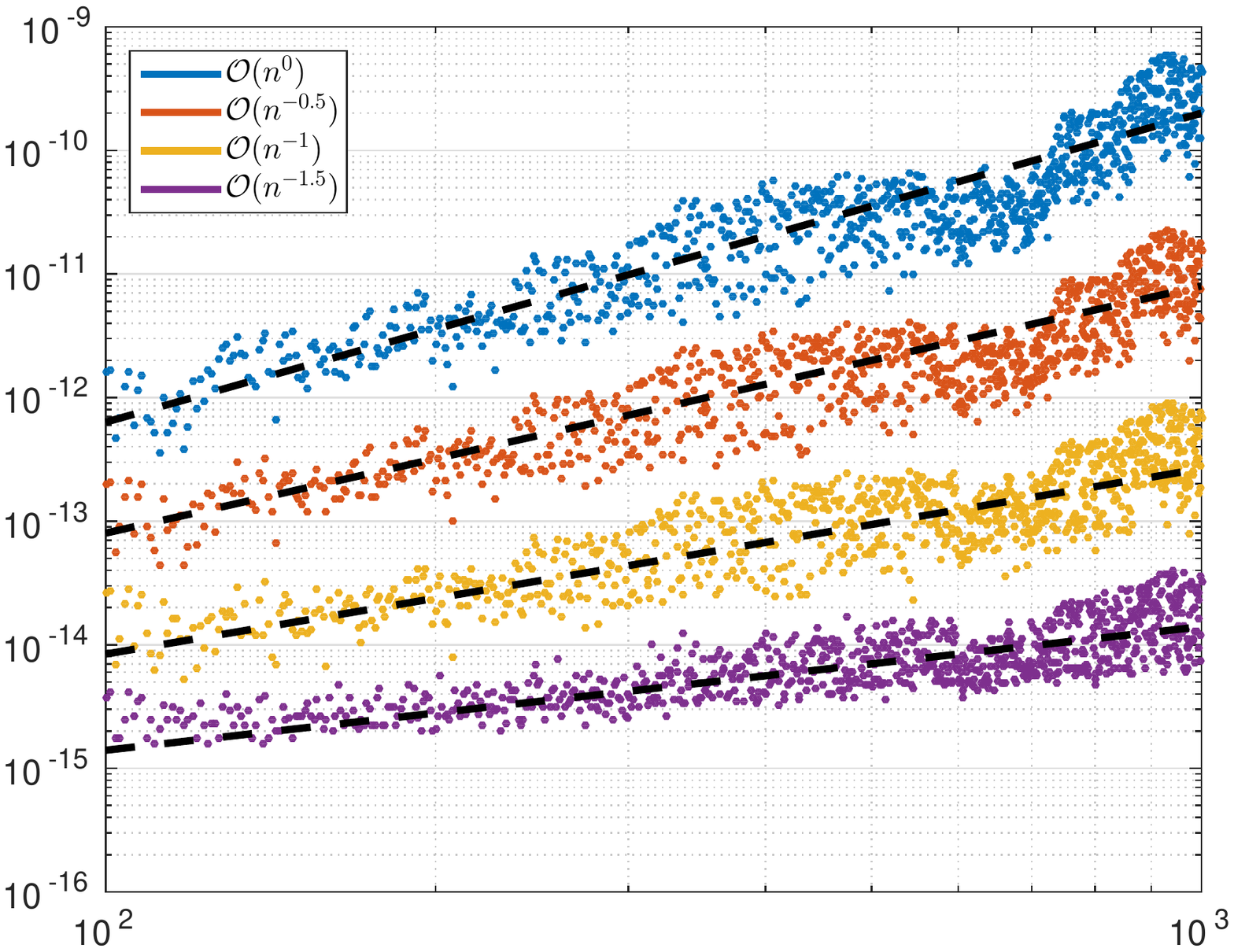}
     \put(37,0) {{$N$}}
     \put(0,20) {{\rotatebox{90}{Absolute error}}}
          \put(0,20) {{\rotatebox{90}{Absolute error}}}
     \put(48,43) {{\rotatebox{17}{\colorbox{white}{$\mathcal{O}(N^{2.5})$}}}}
     \put(48,35) {{\rotatebox{14}{\colorbox{white}{$\mathcal{O}(N^{2})$}}}}
     \put(48,27) {{\rotatebox{10}{\colorbox{white}{$\mathcal{O}(N^{1.5})$}}}}
     \put(48,20) {{\rotatebox{5}{\colorbox{white}{$\mathcal{O}(N)$}}}}
    \end{overpic}\hspace*{5pt}
    \begin{overpic}[width=.49\textwidth]{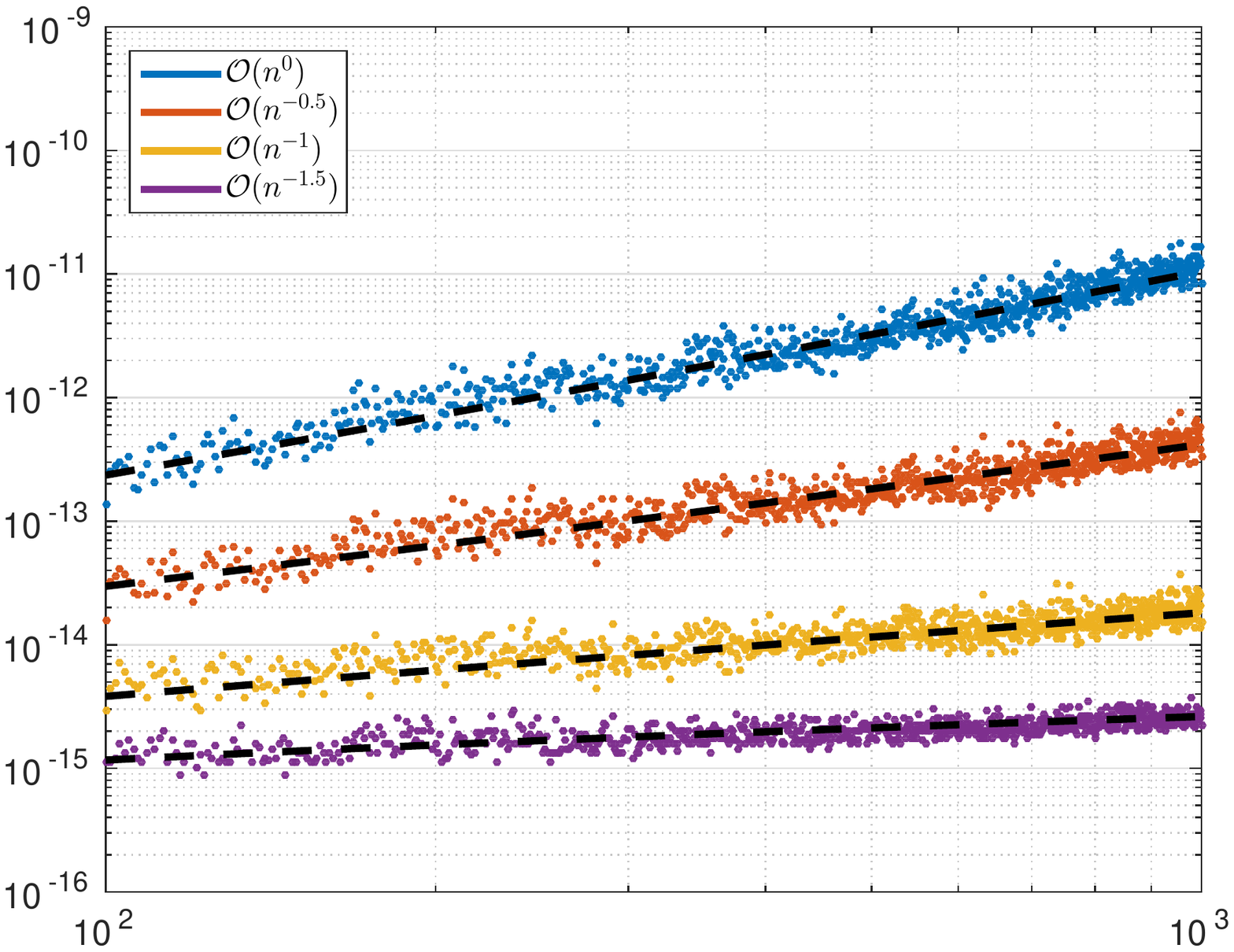}
     \put(37,0) {{$N$}}
     \put(40,35) {{\rotatebox{10}{\colorbox{white}{$\mathcal{O}(N^2/\log^2N)$}}}}
     \put(40,28) {{\rotatebox{8}{\colorbox{white}{$\mathcal{O}(N^{1.5}/\log^2N)$}}}}
     \put(40,20) {{\rotatebox{6}{\colorbox{white}{$\mathcal{O}(N^{0.5}\log N)$}}}}
     \put(40,15) {{\rotatebox{2}{\colorbox{white}{$\mathcal{O}(\log^2N)$}}}}
    \end{overpic}
    \caption{Errors in computing the NDCT of vectors with various rates of decay using the direct method (left) and the new FFT-based approach with $\underline{\theta}^{\ast} = \underline{\theta}^{cheb_1}$ (right). In both cases, 
    a vector of length $N$ is created using {\tt randn(N,1)} in MATLAB and then scaled so that the $n$th entry is $\mathcal{O}(n^0)$, $\mathcal{O}(n^{-0.5})$, $\mathcal{O}(n^{-1})$, and $\mathcal{O}(n^{-1.5})$, 
    giving the four curves in each panel. The dashed lines depict heuristic observations of the error growth in each case. The $\log$ factors may seem somewhat arbitrary, but the
    lines give a much better fit to the data when these are included.}\label{fig:ndct_test01}
\end{figure}

Our second test investigates the time taken to compute a real-valued NDCT of length $N$. Figure \ref{fig:ndct_test02} 
compares the direct approach (squares) and the FFT-based 
approach with $\t^{cheb_1}$ (triangles) and $\t^{cheb_*}$ (circles), in the range $100\leq N\leq 10,\!000$. The variation of computational times between consecutive
values of $N$ is typical of FFT-based algorithms, where the theoretical cost of the FFT as well as the 
precise algorithm employed by FFTW depends on the prime factorization of $N$. The triangles 
and circles lie on a piecewise linear least squares
fit to the computational timings, to suggest some kind of `average' time for near-by values of $N$. 
The variation also makes it difficult to say when the FFT-based algorithm becomes more efficient 
than the direct approach, 
but it occurs around $N=1,\!000$.
\begin{figure}[t]
\vspace*{-6cm}
  \centering
      \begin{overpic}[width=.7\textwidth]{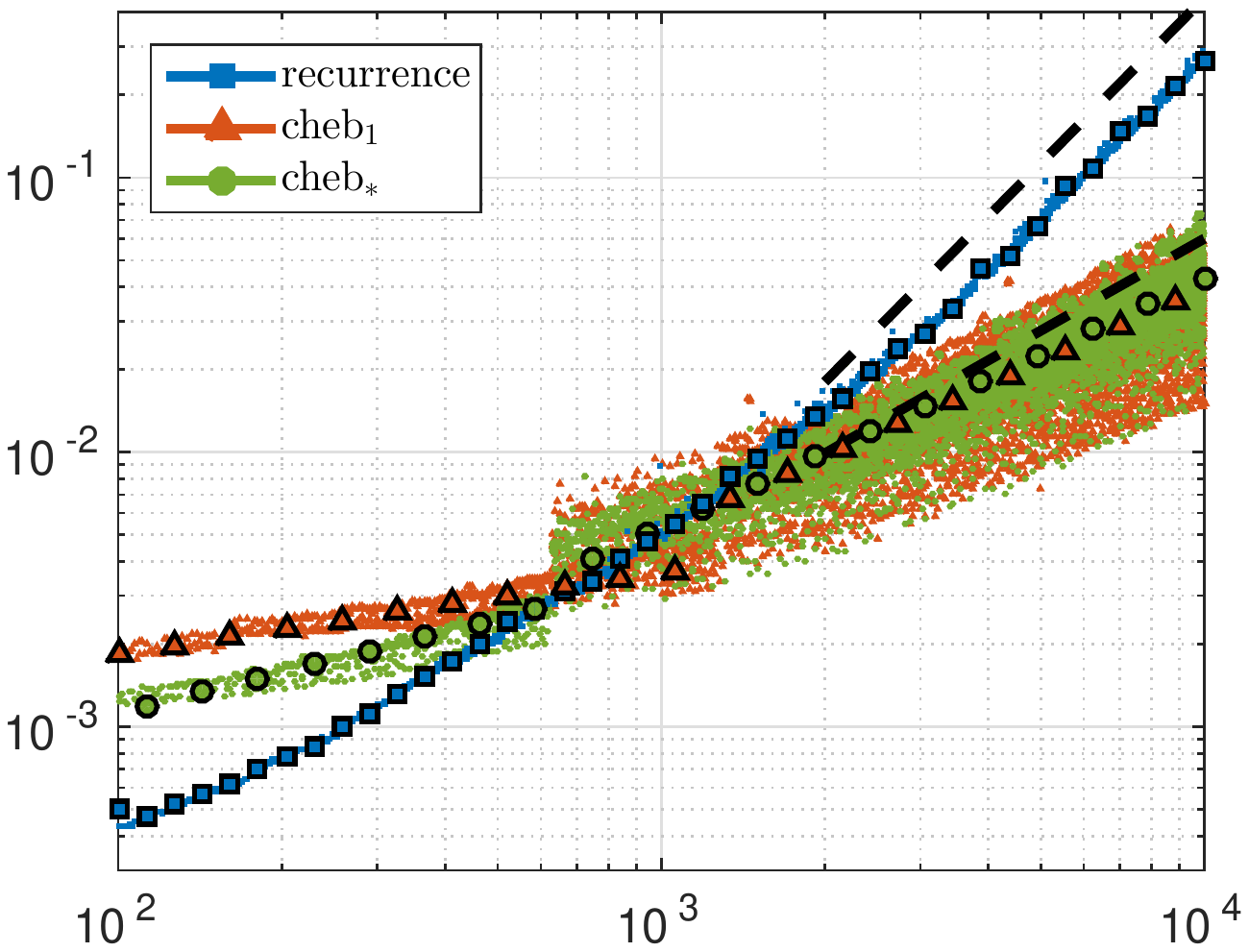}
     \put(30,15) {\scriptsize {$N$}}
     \put(0,30) {\scriptsize {\rotatebox{90}{Time (secs)}}}
     \put(40,47) {\scriptsize {\rotatebox{45}{$\mathcal{O}(N^2)$}}}
     \put(40,40) {\scriptsize {\rotatebox{28}{\colorbox{white!20}{$\mathcal{O}(N\log N)$}}}}
    \end{overpic}\hspace*{-3cm}
     \vspace*{-2cm}
    \caption{Time to compute a NDCT of various lengths using the direct approach (squares) and FFT-based approaches from Section~\ref{sec:dct} with $\underline{x}^{cheb_1}$ (triangles) and $\underline{x}^{cheb_\ast}$ (circles). 
    For larger values of $N$ the two FFT-based approaches require a comparable time, and the crossover with the direct approach occurs at around $N=1,\!000$.
    We believe the jump in the $\underline{x}^{cheb_\ast}$ time (circles) at around $N = 650$ 
    is caused by a switch in algorithmic details of the FFT routine.}\label{fig:ndct_test02}
\end{figure}

%
\section{The discrete Legendre transform}
\label{sec:dlt}
Our procedure for computing the DLT in~\eqref{eq:fastTransform} splits naturally into two stages. The first deals with the non-uniform frequency present 
in $P_n(\cos \theta)$ and converts the transform to one involving $T_n(\cos\theta)$ with uniform frequency in $\theta$. 
The second stage deals with the non-uniform grid $\cos({\t}^{leg}_0), \ldots, \cos(\t_{N-1}^{leg})$ and is the NDCT from the previous section. 

Denote by $M$ the ${N\times N}$ upper-triangular matrix given by
\[
 M_{kn} = \begin{cases} \frac{1}{\pi} \frac{ \Gamma( n/2 + 1/2 ) }{ \Gamma( n/2 + 1 ) }, & 0 = k\leq n \leq N-1,\quad n \text{ even},\\[3pt]
                        \frac{2}{\pi} \frac{ \Gamma( (n-k)/2 + 1/2 ) }{ \Gamma( (n+k)/2 + 1 ) }, & 0 < k\leq n \leq N-1,\quad k+n \text{ even},\\[3pt] 
                        0, & \text{otherwise}.\end{cases} 
\]
If $\underline{c}$ is the vector of Legendre coefficients in~\eqref{eq:fastTransform}
and $\underline{\hat c} = M\underline{c}$, then we have~\cite[]{Alpert_91_01}
 \begin{equation} 
 f_k = \sum_{n=0}^{N-1}c_nP_n(x_k^{leg}) = \sum_{n=0}^{N-1}\hat{c}_nT_n(x_k^{leg}),\qquad 0\leq k \leq N-1.
\label{eq:changeBasis}
\end{equation}
Now, the summation on the righthand side takes the form of the NDCT~\eqref{eq:nuDCT} and
we may use the algorithm described in 
Section~\ref{subsec:dct9} to compute $\underline{f}$ in $\mathcal{O}(N\log N)$ operations.

To compute $\underline{\hat c}$ directly, i.e.\ via the matrix-vector product $M\underline{c}$,  
 requires $\mathcal{O}(N^2)$ operations. Instead, we note that the transform 
$\underline{\hat c} = M\underline{c}$ 
can be 
computed in $\mathcal{O}(N(\log N)^2/\log\!\log N)$ operations using the 
FFT-based algorithm described in~\cite[]{Hale_14_01} or by the fast multipole-based 
approach in~\cite[]{Alpert_91_01}.  In this paper we use the algorithm in~\cite[]{Hale_14_01} 
because it has no precomputational cost and so allows for adaptive discretizations. 

Another convenient property of the transforms in this paper is that there is an asymptotically optimal
selection of algorithmic parameters for any working tolerance. In the NDCT, the number of terms in the 
Taylor series expansion can be precisely determined based on the working tolerance. 
We would like this to also hold for the $\mathcal{O}(N(\log N)^2/\log\!\log N)$ 
Chebyshev--Legendre transform described in~\cite[]{Hale_14_01}.  We remark that though
the paper considered only a working tolerance of machine precision throughout, the 
algorithmic parameters were derived and determined in terms of $\varepsilon>0$. We  
slightly modified our implementation to exploit arbitrary working tolerances. 
%
%
\subsection{Numerical results for the discrete Legendre transform}\label{subsec:dlt_results}
First, we take random vectors $\underline{c}$ with normally distributed entries and then scale
them to have algebraic rates of decay $\mathcal{O}(n^0)$, $\mathcal{O}(n^{-0.5})$, $\mathcal{O}(n^{-1})$, and $\mathcal{O}(n^{-1.5})$, as in 
Section~\ref{subsec:NumericalResults}. The accuracy of both the direct and FFT-based approaches is compared 
against an extended precision computation. The results are shown in Figure~\ref{fig:dlt_err} and one can see that 
the errors for the direct and FFT-based approach are comparable. In most practical applications the Legendre coefficients 
are associated to a smooth function so they decay at least algebraically. 
\begin{figure}[t]
\vspace*{-4cm}
  \centering
  \scriptsize
    \begin{overpic}[width=.49\textwidth]{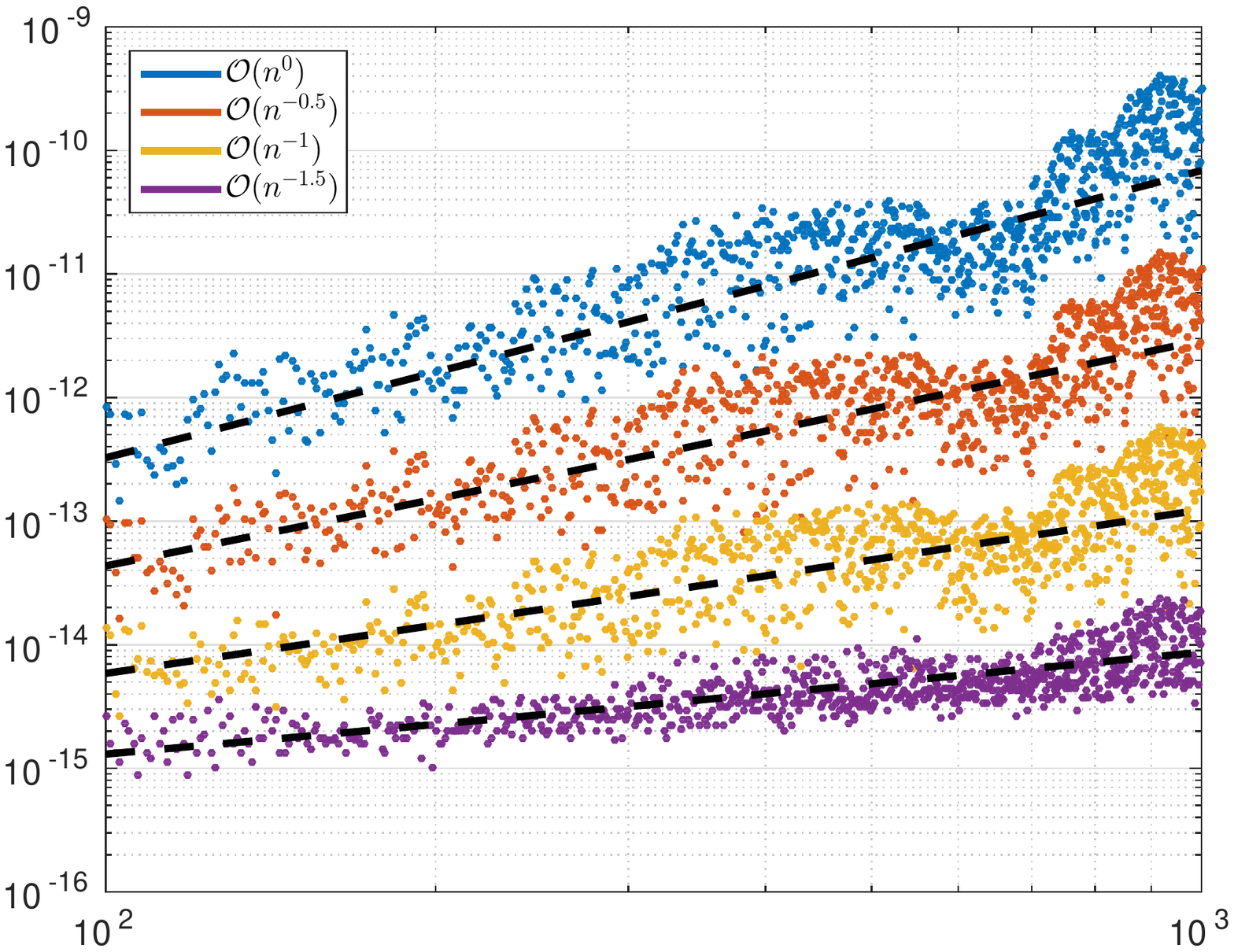}
     \put(37,0) {{$N$}}
          \put(0,22) {{\rotatebox{90}{Absolute error}}}
     \put(48,41) {{\rotatebox{14}{\colorbox{white}{$\mathcal{O}(N^{2.5}/\log N)$}}}}
     \put(48,32) {{\rotatebox{12}{\colorbox{white}{$\mathcal{O}(N^{2}/\log N)$}}}}
     \put(48,25) {{\rotatebox{9}{\colorbox{white}{$\mathcal{O}(N^{1.5}/\log N)$}}}}
     \put(48,18) {{\rotatebox{5}{\colorbox{white}{$\mathcal{O}(N/\log N)$}}}}
    \end{overpic}\hspace*{5pt}
    \begin{overpic}[width=.49\textwidth]{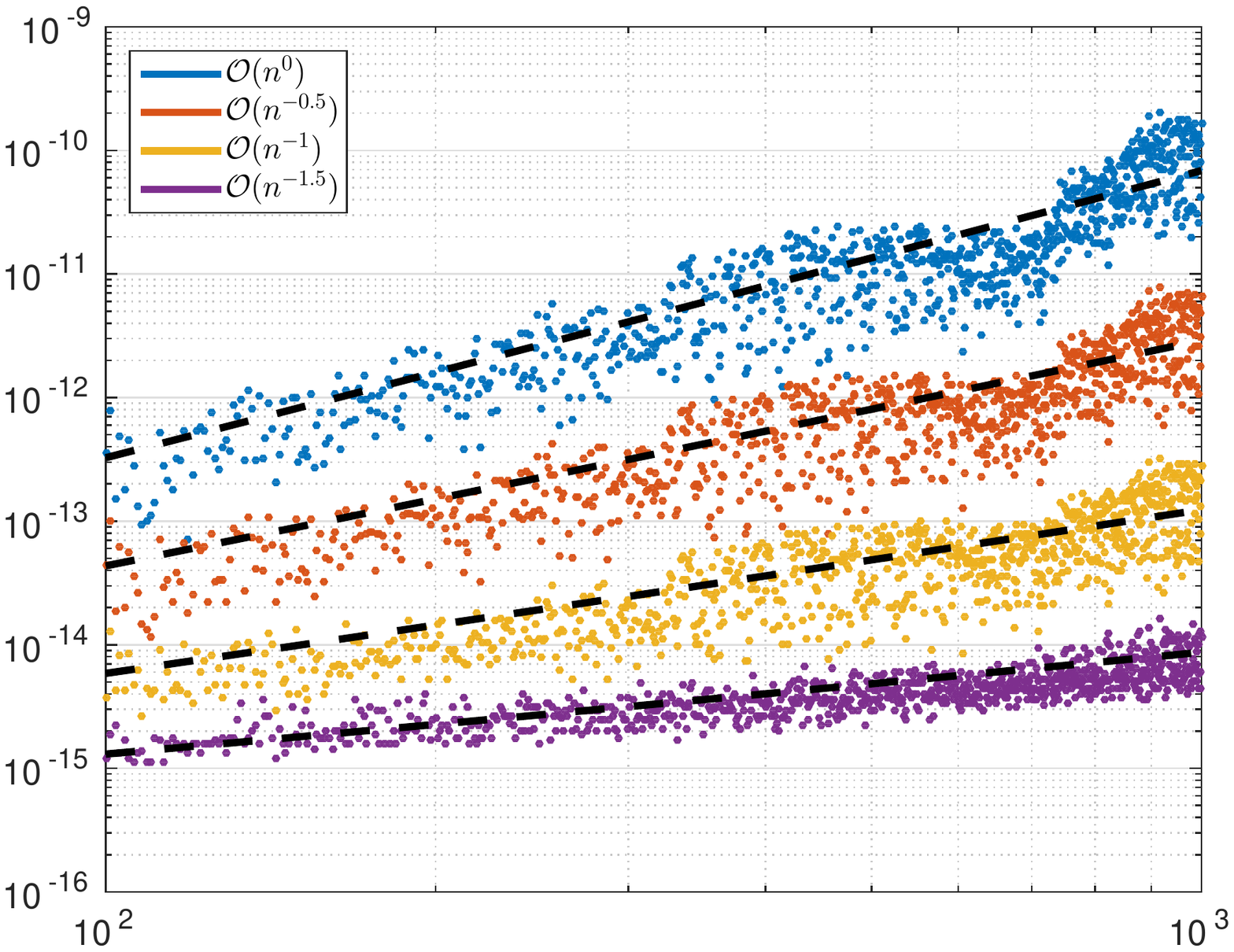}
     \put(37,0) {{$N$}}
     \put(48,41) {{\rotatebox{14}{\colorbox{white}{$\mathcal{O}(N^{2.5}/\log N)$}}}}
     \put(48,32) {{\rotatebox{12}{\colorbox{white}{$\mathcal{O}(N^{2}/\log N)$}}}}
     \put(48,25) {{\rotatebox{9}{\colorbox{white}{$\mathcal{O}(N^{1.5}/\log N)$}}}}
     \put(48,18) {{\rotatebox{5}{\colorbox{white}{$\mathcal{O}(N/\log N)$}}}}
    \end{overpic}
    \caption{Errors in computing the DLT of vectors with various rates of decay using the direct approach (left) and the FFT-based approach (right). In both cases, 
    a vector of length $N$ is created using {\tt randn(N,1)} in MATLAB and then scaled so that the $n$th entry is $\mathcal{O}(n^0)$, $\mathcal{O}(n^{-0.5})$, $\mathcal{O}(n^{-1})$, and $\mathcal{O}(n^{-1.5})$, 
    giving the four curves in each panel. The dashed lines depict heuristic observations of the error growth in each case. 
}\label{fig:dlt_err}
\end{figure}

While we achieve a similar error to the direct approach, our FFT-based approach has a lower complexity and are more computationally efficient for large $N$. Figure~\ref{fig:dlt_time} shows 
a comparison of the computational times for the direct and FFT-based approach.  When comparing with Figure~\ref{fig:ndct_test02} we see 
that the computational time for the DLT is dominated by the \lstinline{leg2cheb} transformation. For an accuracy goal of double 
precision, the crossover between the direct and FFT-based approach is approximately $N = 5,\!000$.%
\begin{figure}[t]
\vspace{-5cm}
  \centering
  \scriptsize
    \begin{overpic}[width=.6\textwidth]{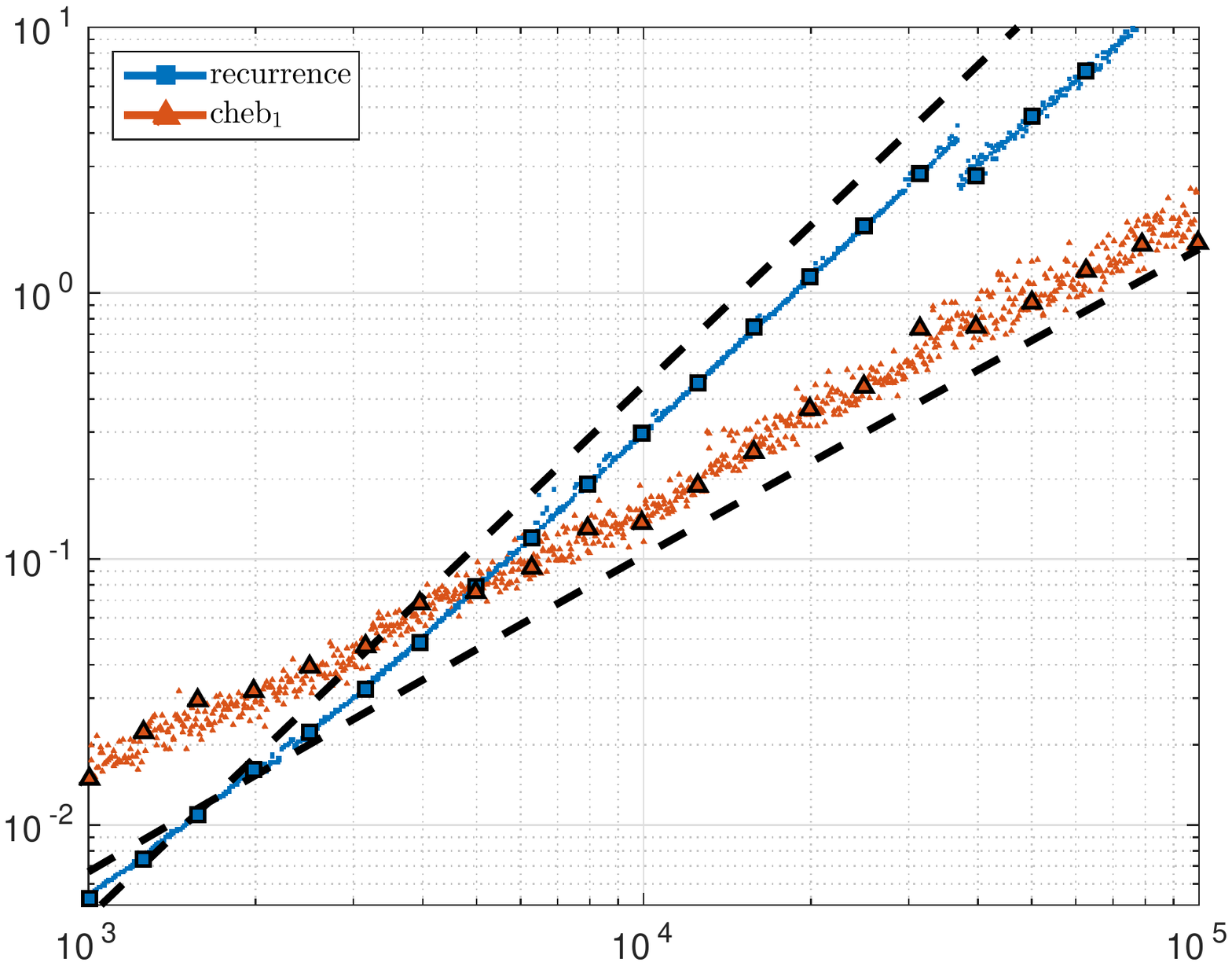}
     \put(37,0) {{$N$}}
     \put(0,25) {{\rotatebox{90}{Time (secs)}}}
     \put(47,44) {{\rotatebox{43}{{$\mathcal{O}(N^2)$}}}}
     \put(42,23) {{\rotatebox{29}{{$\mathcal{O}(N(\log N)^2/\log\log N)$}}}}
    \end{overpic}
    \caption{Time taken to compute the DLT of length $N$ using the direct approach (squares) and FFT-based approach (triangles). 
    Here we choose 1000 logarithmically spaced values of $N$ in the range $10^3$--$10^5$. The larger squares and rectangles are every 50th one of these.
    We omit the results for the FFT-based approach using the points $\underline{x}^{cheb_\ast}$ as they are indistinguishable 
    from those of $\underline{x}^{cheb_1}$.
  }\label{fig:dlt_time}
\end{figure}
%
\section{The inverse discrete Legendre transform}
\label{sec:idlt}
We now turn to the IDLT, which can be seen as taking values of a function on 
an $N$-point Legendre grid and returning the corresponding Legendre coefficients 
of the degree $N-1$ polynomial interpolant. The basis of the algorithm we propose 
is analogous to the forward transform and has the same $\mathcal{O}(N(\log N)^2/\log\!\log N)$ complexity. 
Unlike most NFFT algorithms which resort to an iterative approach for 
computing the inverse transform \cite[]{anderson_96_01, Dutt_93_01},
here we take advantage of the orthogonality of Legendre polynomials, and in particular 
the exactness of Gauss--Legendre quadrature, to derive a direct method for the IDLT.

It is convenient to work with matrix notation rather than summations, and we denote the $N\times N$ matrix with $(j,k)$ entries $P_k(x_j^{leg})$ as 
$\mathbf{P}_N(\underline{x}^{leg})$ and the $N\times N$ with $(j,k)$ entry $T_k(x_j^{leg})$ 
as $\mathbf{T}_N(\underline{x}^{leg})$.
Using the orthogonality of Legendre polynomials and the exactness of Gauss--Legendre quadrature, the IDLT~\eqref{eq:fastTransformInv} can then be conveniently written as
\begin{equation}\label{eqnn:cinmatform}
\underline{c} = \mathbf{P}_N^{-1}(\underline{x}^{leg})\underline{f} = D_{\underline{s}}\mathbf{P}_N^T(\underline{x}^{leg})D_{\underline{w}^{leg}}\underline{f},
\end{equation}
where $\underline{s} = (\|P_0\|_2^{-2}, \ldots, \|P_{N-1}\|_2^{-2})^T$ and $\underline{w}^{leg}$ is the vector of Gauss--Legendre quadrature weights. 
From~\eqref{eq:changeBasis} we know that, for any vector $\underline{c}$, 
$\mathbf{P}_N(\underline{x}^{leg})\underline{c} = \mathbf{T}_N(\underline{x}^{leg})\hat{\underline{c}} = \mathbf{T}_N(\underline{x}^{leg})M{\underline{c}}$,
and hence we have the matrix decomposition $\mathbf{P}_N(\underline{x}^{leg}) = \mathbf{T}_N(\underline{x}^{leg})M$.
Substituting this into~\eqref{eqnn:cinmatform} we obtain
\[
 \underline{c} = D_{\underline{s}}M^T\mathbf{T}^{T}_N(\underline{x}^{leg})D_{\underline{w}}\underline{f}.
\]
Thus, $\underline{c}$ can also be computed in $\mathcal{O}(N(\log N)^2/\log\!\log N)$ operations because:
\begin{enumerate}
 \item $D_{\underline{w}}$ and $D_{\underline{s}}$ are diagonal matrices, so can be readily be applied to a vector in $\mathcal{O}(N)$ operations,
 \item $\mathbf{T}_N^T(\underline{x}^{leg})$ is precisely the transpose of the NDCT in Section~\ref{sec:dct}. Since the transpose of the DCTs and DSTs
 in Section~\ref{subsec:dct9} can themselves be expressed in terms of DCTs and DSTs (of type-II), $\mathbf{T}_N^T(\underline{x}^{leg})$ can be approximated 
 in $\mathcal{O}(N\log N)$ operations by taking the transpose of~\eqref{eqn:lincomb}.
 \item 
 The matrix-vector product with $M^T$ is closely related to the
 Chebyshev--Legendre transform and can be computed in $\mathcal{O}(N(\log N)^2/\log\!\log N)$ 
 operations by a slight modification of the algorithm in~\cite[]{Hale_14_01}. (Alternatively, one could use the ideas in~\cite[]{Alpert_91_01}.)
\end{enumerate}

\subsection{Numerical results for the inverse discrete Legendre transform} 
%
For the inverse transform, the direct approach requires $\mathcal{O}(N^2)$ operations
to compute the IDLT of length $N$, rather than the naive $\mathcal{O}(N^3)$ standard inversion algorithm. 
This is because the direct approach we have implemented exploits the discrete orthogonality 
relation that holds for Legendre polynomials (see~\eqref{eqnn:cinmatform}). The matrix-vector 
product with $\mathbf{P}_N^T(\underline{x}^{leg})$ in~\eqref{eqnn:cinmatform} is computed 
by using the three-term recurrence relations satisfied by Legendre polynomials, except now along 
rows instead of columns. 

For our first numerical experiment we take randomly generated vectors with normally distributed entries 
and compare the accuracy of both the direct and FFT-based approach against 
an extended precision computation (see Figure~\ref{fig:lastfig}, left). Unlike in Section~\ref{subsec:NumericalResults} 
and Section~\ref{subsec:dlt_results} we do not consider decay in these vectors, as they will typically correspond to function
values in space where there is no reason to expect decay. We find that the errors incurred by our FFT-based IDLT 
are consistent with the direct approach.

In Figure~\ref{fig:lastfig} (right) we show the execution time for computing the IDLT with 
the direct approach (squares) and FFT-based approach (triangles). As with the DLT, our IDLT algorithm is more computationally 
efficient than the direct approach when $N\geq 5,\!000$. 
\begin{figure} 
\vspace{-4cm}
  \centering
  \scriptsize
    \begin{overpic}[width=.49\textwidth]{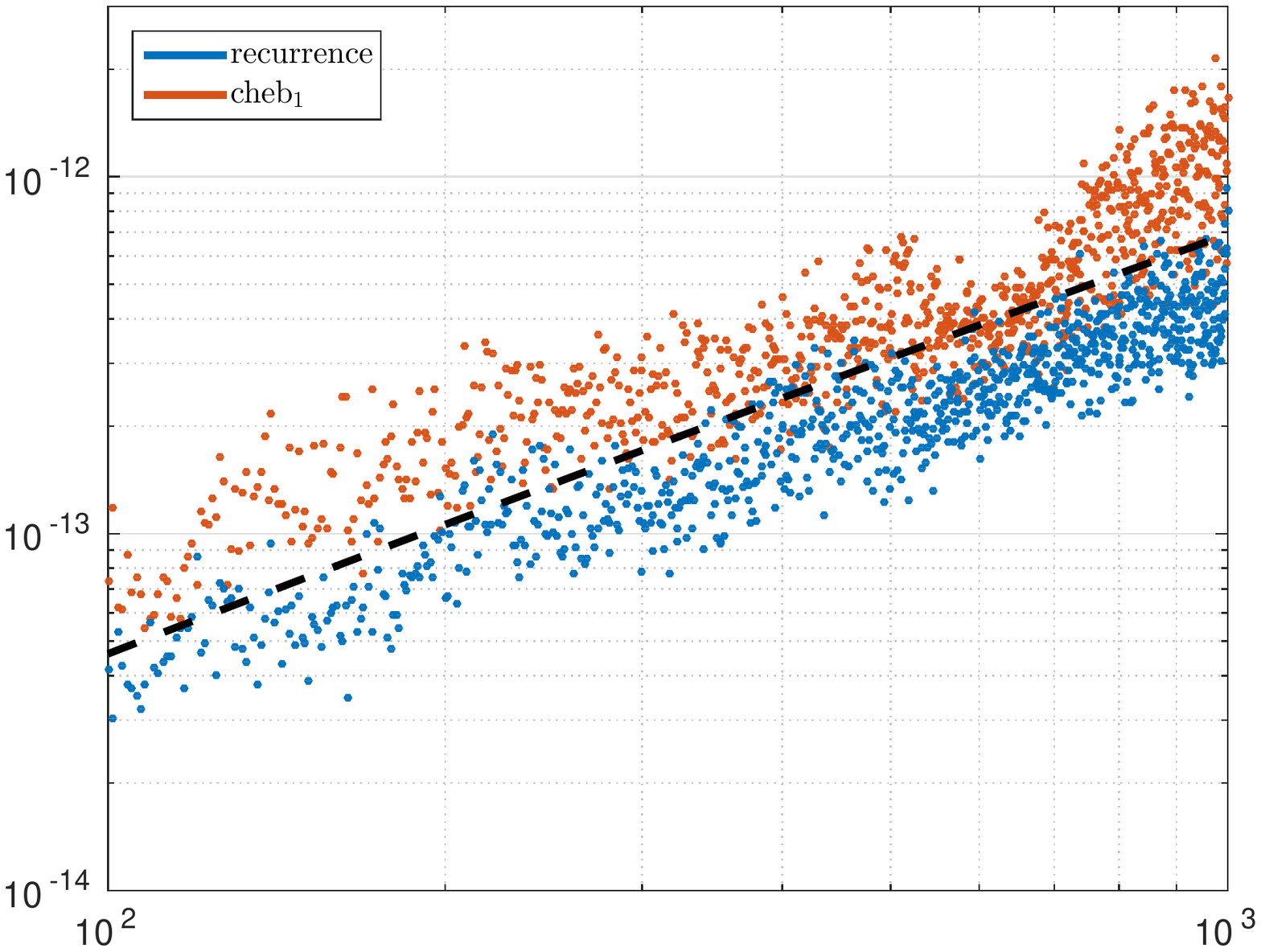}
     \put(37,0) {{$N$}}
     \put(0,22) {{\rotatebox{90}{Absolute error}}}
     \put(48,41) {{\rotatebox{20}{\colorbox{white}{$\mathcal{O}(N \log N)$}}}}
    \end{overpic}\hspace*{5pt}
    \begin{overpic}[width=.483\textwidth]{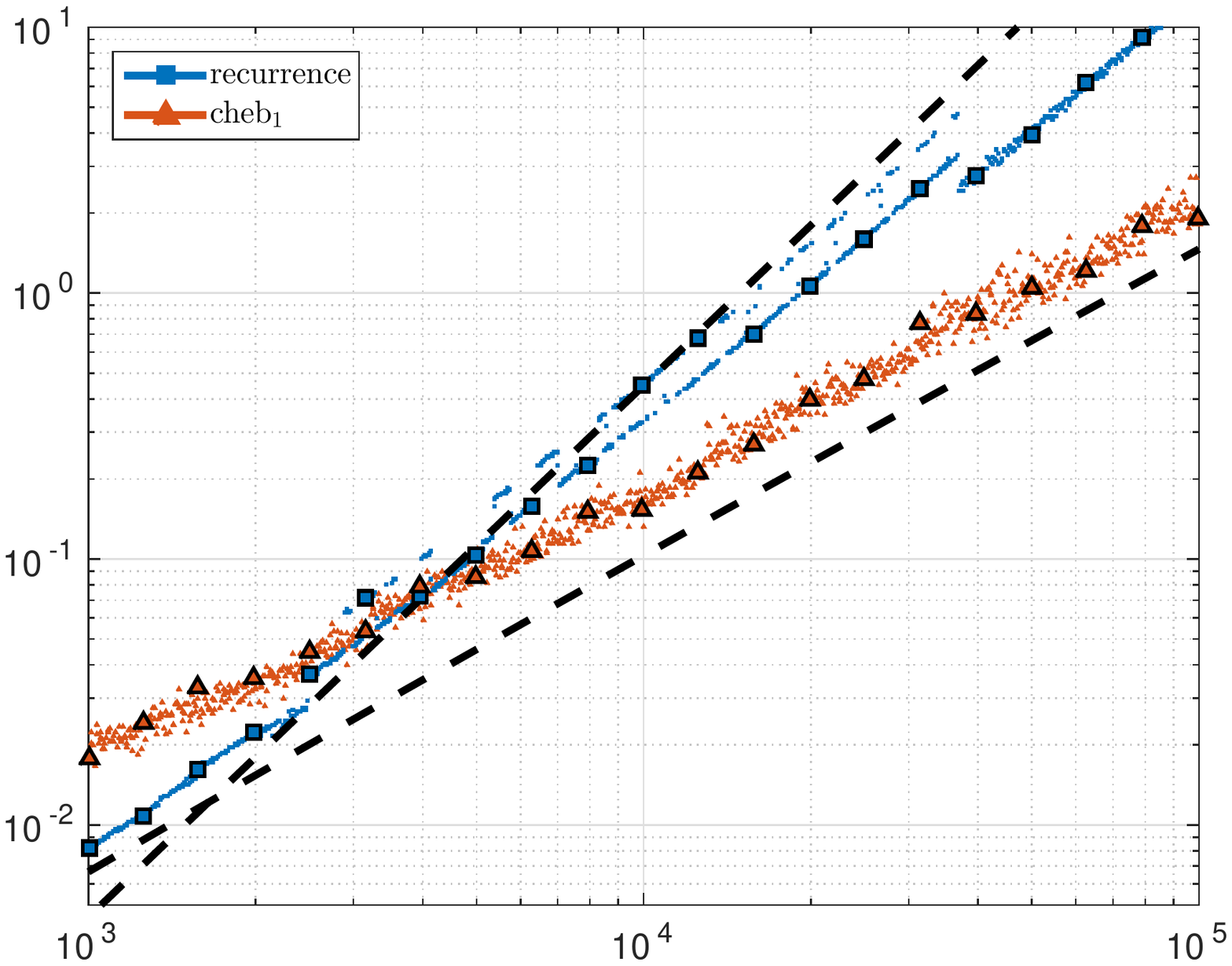}
     \put(37,0) {{$N$}}
     \put(0,22) {{\rotatebox{90}{Time (secs)}}}
     \put(45,43) {{\rotatebox{43}{{$\mathcal{O}(N^2)$}}}}
     \put(40,22) {{\rotatebox{29}{{$\mathcal{O}(N(\log N)^2/\log\log N)$}}}}
    \end{overpic}
    \caption{Left: Errors in computing the IDLT of vectors using the direct method and the FFT-based approach. 
    A vector of length $N$ is created using {\tt randn(N,1)} in MATLAB with the dashed line showing the observed
    growth of the maximum absolute error. 
    Right: Time taken to compute the IDLT of length $N$ using the direct approach (squares) and FFT-based approach (triangles). }
    \label{fig:lastfig}
\end{figure}

%
%
%
\section*{Conclusion}
\label{sec:conc}
We have presented fast and simple algorithms for the discrete Legendre 
transform and its inverse, which rely on a nonuniform discrete cosine transform 
and the Chebyshev--Legendre transform. Both components are based on the 
fast Fourier transform and have no precomputational cost. For an $N$-point 
transform both algorithms have a complexity of $\mathcal{O}(N(\log N)^2/\log\log N)$ operations
and are faster than the direct approach when $N\geq 5,\!000$. MATLAB code
to compute the transformations (and all the results contained in this paper)
are available online \cite[]{Hale_15_01}. The codes are also accessible
in Chebfun \cite[]{chebfun} via \lstinline{chebfun.dlt()} and \lstinline{chebfun.idlt()}, respectively.
As part of the analysis of the algorithm we derived a bound on the 
distance between associated points in a Chebyshev and Legendre grid
of the same size (see Lemma~\ref{lemma:szego}), which we believe to be tighter than 
any similar bounds appearing in the literature.

\section*{Acknowledgments}
We thank Andr\'{e} Weideman and Anthony Austin for reading a draft version of this 
manuscript, and the anonymous referees for their useful suggestions. 

\renewcommand{\UrlFont}{\small\tt}
\bibliographystyle{IMANUM-BIB}
\bibliography{haletownsend2015_revised}

\end{document}